\documentclass{amsart}
\usepackage{amsmath,amssymb}
\usepackage[dvips]{graphics}
\usepackage[all]{xy}
\newtheorem{Theorem}{Theorem}[section]
\newtheorem{Lemma}[Theorem]{Lemma}
\newtheorem{Proposition}[Theorem]{Proposition}

\theoremstyle{Definition}

\newtheorem{Example}[Theorem]{Example}

\theoremstyle{Remark}

\makeatletter
\@addtoreset{figure}{section}
\def\@thmcountersep{-}
\makeatother

\numberwithin{equation}{section}

\begin{document} 

\title{An intrinsic non-triviality of graphs}

\author{Ryo Nikkuni}
\address{Institute of Human and Social Sciences, Faculty of Teacher Education, Kanazawa University, Kakuma-machi, Kanazawa, Ishikawa, 920-1192, Japan}
\email{nick@ed.kanazawa-u.ac.jp}
\thanks{The author was partially supported by Grant-in-Aid for Young Scientists (B) (No. 18740030), Japan Society for the Promotion of Science.}


\date{}

\dedicatory{Dedicated to Professor Akio Kawauchi for his 60th birthday}


\begin{abstract}
We say that a graph is intrinsically non-trivial if every spatial embedding of the graph contains a non-trivial spatial subgraph. We prove that an intrinsically non-trivial graph is intrinsically linked, namely every spatial embedding of the graph contains a non-splittable $2$-component link. We also show that there exists a graph such that every spatial embedding of the graph contains either a non-splittable $3$-component link or an irreducible spatial handcuff graph whose constituent $2$-component link is split. 
\end{abstract}

\maketitle

\section{Introduction} 

Throughout this paper we work in the piecewise linear category. Let $f$ be an embedding of a finite graph $G$ into the $3$-sphere ${\mathbb S}^{3}$. Then $f$ (or $f(G)$) is called a {\it spatial embedding} of $G$ or simply a {\it spatial graph}. We call a subgraph $\gamma$ of $G$ which is homeomorphic to the circle a {\it cycle}. If $G$ is homeomorphic to the disjoint union of cycles, then $f$ is none other than an {\it $n$-component link} (or {\it knot} if $n=1$). Two spatial embeddings $f$ and $g$ of $G$ are said to be {\it ambient isotopic} if there exists an orientation-preserving self homeomorphism $\Phi$ on ${\mathbb S}^{3}$ such that $\Phi\circ f=g$. A graph $G$ is said to be {\it planar} if there exists an embedding of $G$ into the $2$-sphere, and a spatial embedding $f$ of a planar graph $G$ is said to be {\it trivial} if it is ambient isotopic to an embedding of $G$ into a $2$-sphere in ${\mathbb S}^{3}$. A spatial embedding $f$ of $G$ is said to be {\it split} if there exists a $2$-sphere $S$ in ${\mathbb S}^{3}$ such that $S\cap f(G)=\emptyset$ and each connected component of ${\mathbb S}^{3}\setminus S$ has intersection with $f(G)$, and otherwise $f$ is said to be {\it non-splittable}. 

A graph $G$ is said to be {\it intrinsically linked} if every spatial embedding $f$ of $G$ contains a non-splittable $2$-component link. Conway-Gordon \cite{CG83} and Sachs \cite{Sachs84} showed that $K_{6}$ is intrinsically linked, where $K_{n}$ denotes the {\it complete graph} on $n$ vertices. Conway-Gordon also showed that $K_{7}$ is {\it intrinsically knotted}, namely every spatial embedding $f$ of $G$ contains a non-trivial knot. For a positive integer $n$, Flapan-Foisy-Naimi-Pommersheim \cite{FFNP} showed that there exists an {\it intrinsically $n$-linked} graph $G$, namely every spatial embedding $f$ of $G$ contains a non-splittable $n$-component link (see also \cite{FNP01}, \cite{BF04} for the case of $n=3$). Note that these results paid attention to only constituent knots and links of spatial graphs. Our purpose in this paper is to generalize the notion of intrinsically $n$-linkedness by paying attention to not only constituent knots and links but also spatial subgraphs which do not need to be knots and links,  and to give graphs which have such a generalized property. We say that a graph $G$ is {\it intrinsically non-trivial} if for every spatial embedding $f$ of $G$ there exists a planar subgraph $F$ of $G$ such that $f|_{F}$ is not trivial. It is clear that intrinsically $n$-linked graphs are intrinsically non-trivial. Note that $F$ depends on $f$ and does not need to have the same topological type uniformly. In this situation, a pioneer work has been done by Foisy \cite{Foisy06}. Let $P$ and $P'$ be graphs in the {\it Petersen family}, which is a family of seven graphs obtained from $K_{6}$ by {\it $\nabla - Y$} or {\it $Y - \nabla$ exchanges} \cite{Sachs84}, and all intrinsically linked graphs have a minor in this family \cite{RST3}. Let $P {*}_{4}P'$ be the graph which consists of $P$ and $P'$ connected by four disjoint edges $e_{1},e_{2},e_{3}$ and $e_{4}$ as illustrated in Fig. \ref{K6K6_4} (Fig. \ref{K6K6_4} shows the case of $P=P'=K_{6}$). Then he showed that $P {*}_{4}P'$ is {\it intrinsically knotted or $3$-linked}, namely every spatial embedding of $P{*}_{4}P'$ contains either a non-trivial knot or a non-splittable $3$-component link. He also showed that $K_{6} {*}_{4}K_{6}$ is neither intrinsically knotted nor intrinsically $3$-linked.  
\begin{figure}[htbp]
      \begin{center}
\scalebox{0.4}{\includegraphics*{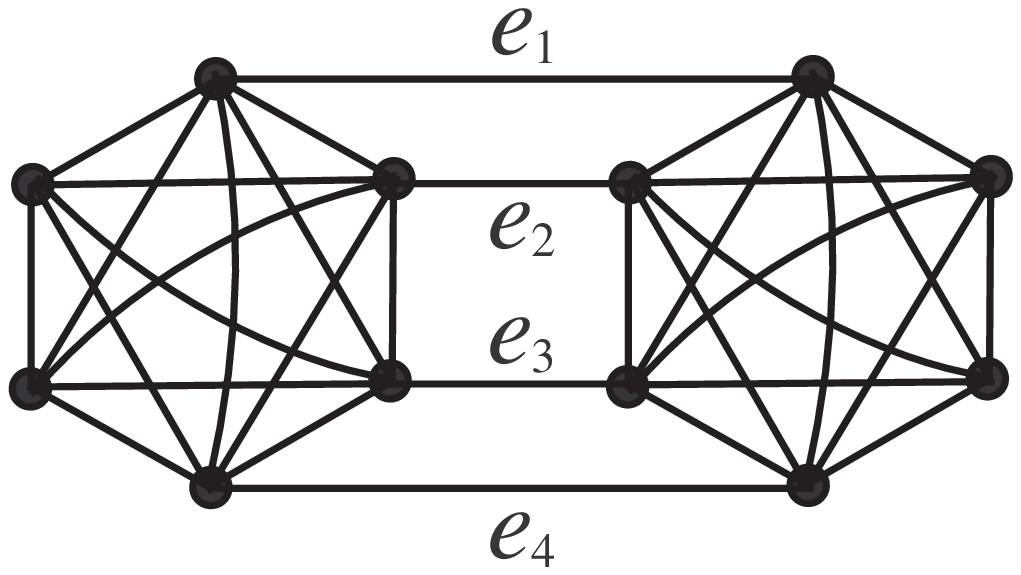}}
      \end{center}
   \caption{$K_{6}{*}_{4}K_{6}$}
  \label{K6K6_4}
\end{figure} 

We have already known that intrinsically knotted graphs and intrinsically $n$-linked graphs ($n\ge 3$) are intrinsically linked by Robertson-Seymour-Thomas' characterization of intrinsically linked graphs \cite{RST3}. First of all, we show that if a graph $G$ is intrinsically non-trivial then every spatial embedding of $G$ must contain a non-splittable $2$-component link as follows. 

\begin{Theorem}\label{int_2}
Intrinsically non-trivial graphs are intrinsically linked. 
\end{Theorem}

Then, for an intrinsically non-trivial graph $G$, we are interested in the collection of non-trivial spatial graph types (except for $2$-component link types) where every spatial embedding of $G$ contains at least one spatial graph type in the collection. Foisy's above result also paid attention to only constituent knots and links of spatial graphs. We shall consider more smaller graph than $P {*}_{4} P'$ from a viewpoint of an intrinsic non-triviality. Let $P {*}_{3} P'$ be the graph which is obtained from $P {*}_{4} P'$ by deleting $e_{4}$. Then we have the following. 

\begin{Theorem}\label{main}
Let $P$ and $P'$ be graphs in the Petersen family. Then, every spatial embedding of $P{*}_{3}P'$ contains either a non-splittable $3$-component link or an irreducible spatial handcuff graph whose constituent $2$-component link is split. 
\end{Theorem}

Here a {\it spatial handcuff graph} is a spatial embedding $f$ of the graph $H$ which is illustrated in Fig. \ref{handcuff}. Note that an orientation is given to each loop, namely we regard $f(\gamma_{1}\cup \gamma_{2})$ as an ordered and oriented $2$-component link. A spatial handcuff graph $f$ is said to be {\it irreducible} if there does not exist a $2$-sphere in ${\mathbb S}^{3}$ which intersects $f(H)$ transversely at one point. Note that an irreducible spatial handcuff graph is not trivial. We also show that $K_{6} {*}_{3} K_{6}$ have a spatial embedding which does not contain a non-splittable $3$-component link, and another spatial embedding which does not contain an irreducible spatial handcuff graph whose constituent $2$-component link is split (Example \ref{ex}). In particular, the former spatial embedding does not contain any of a non-trivial knot or a non-splittable $n$-component link for $n\ge 3$. Namely Theorem \ref{main} gives a new type of intrinsic non-triviality of graphs which cannot be detected by observing only constituent knots and links of its spatial embeddings. 
\begin{figure}[htbp]
      \begin{center}
\scalebox{0.4}{\includegraphics*{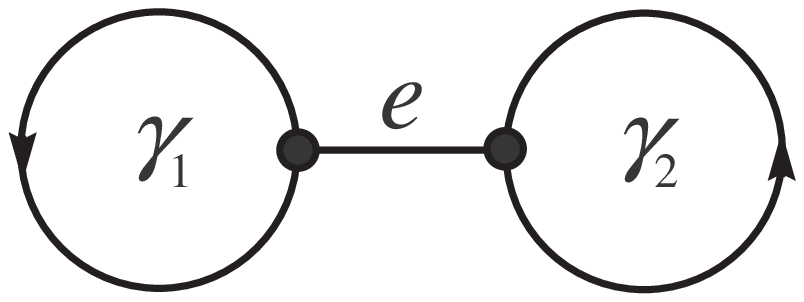}}
      \end{center}
   \caption{}
  \label{handcuff}
\end{figure} 

In the next section, we prove Theorem \ref{int_2}. In section $3$, we recall an ambient isotopy invariant of spatial handcuff graphs which was introduced by the author in \cite{Nikkuni08}. In section $4$, we prove Theorem \ref{main}. 

\section{Proof of Theorem \ref{int_2}} 

A spatial embedding $f$ of a graph $G$ is said to be {\it free} if the fundamental group of the spatial graph complement $\pi_{1}({\mathbb S}^{3}\setminus f(G))$ is a free group. We say that $G$ is {\it intrinsically non-free} if for every spatial embedding of $G$ there exists a subgraph $F$ of $G$ such that $f|_{F}$ is not free. To prove Theorem \ref{int_2}, we show the following. 

\begin{Theorem}\label{int_2_lemma}
The following are equivalent. 
\begin{enumerate}
\item $G$ is intrinsically non-free. 
\item $G$ is intrinsically non-trivial. 
\item $G$ is intrinsically linked. 
\end{enumerate}
\end{Theorem}

\begin{proof}
It is clear that (3) implies (2). Next we show that (2) implies (1). If $G$ is intrinsically non-trivial, we have that for every spatial embedding $f$ of $G$ there exists a planar subgraph $F$ of $G$ such that $f|_{F}$ is non-trivial. Then by Scharlemann-Thompson's famous criterion \cite{ST91}, there exists a subgraph $F'$ of $F$ such that $f|_{F'}$ is not free. Since $F'$ is also a subgraph of $G$, we have that $G$ is intrinsically non-free. Finally we show that (1) implies (3). Assume that $G$ is not intrinsically linked. Then it follows from Robertson-Seymour-Thomas \cite[(1.2)]{RST3} that there exists a spatial embedding $f$ of $G$ such that for any cycle $\gamma$ of $G$ there exists a $2$-disk $D_{\gamma}$ in ${\mathbb S}^{3}$ such that $f(G)\cap D_{\gamma}=f(G)\cap \partial D_{\gamma}=f(\gamma)$. At this time, it is  also known that $f|_{F}$ is free for any subgraph $F$ of $G$ \cite[(3.3)]{RST3} (the case that $G$ is planar was first shown by Wu \cite{Wu92}). Thus we have that $G$ is not intrinsically non-free. 
\end{proof}

\section{An invariant of spatial handcuff graphs} 

In this section we give the definition of an invariant of spatial handcuff graphs which can detect an irreducible one whose constituent $2$-component link is split. Let $L=J_{1}\cup J_{2}$ be an ordered and oriented $2$-component link. Let $D$ be an oriented $2$-disk and $x_{1},x_{2}$ disjoint arcs in $\partial D$, where $\partial D$ has the orientation induced by the one of $D$, and each arc has an orientation induced by the one of $\partial D$. We assume that $D$ is embedded in ${\mathbb S}^{3}$ so that $D \cap L=x_{1}\cup x_{2}$ and $x_{i}\subset J_{i}$ with opposite orientations for each $i$. Then we call a knot $K_{D}=\left(L\cup \partial D\right)\setminus \left({\rm int}x_{1}\cup {\rm int}x_{2}\right)$ a {\it $D$-sum} of $L$. For a spatial handcuff graph $f$, we denote $f(\gamma_{1}\cup \gamma_{2})$ by $L_{f}$ and consider a $D$-sum of $L_{f}$ so that $f(e)\subset D$, $f(e)\cap \partial D=f(e)\cap L_{f}=\left\{p_{1},p_{2}\right\}$ and $p_{i}\in {\rm int}x_{i}\ (i=1,2)$. We call such a $D$-sum of $L_{f}$ a {\it $D$-sum of $L_{f}$ with respect to $f$} and denote it by $K_{D}(f)$. Though $K_{D}(f)$ is not uniquely determined up to ambient isotopy, the author showed in \cite[Remark 3.4 (1)]{Nikkuni08} that the modulo ${\rm lk}(L_{f})$ reduction of $a_{2}(K_{D}(f))$ is an ambient isotopy invariant of $f$, where ${\rm lk}$ denotes the {\it linking number} in ${\mathbb S}^{3}$ and $a_{2}$ denotes the second coefficient of the {\it Conway polynomial}. Then we define that 
\begin{eqnarray*}
n(f,D)=a_{2}(K_{D}(f))-a_{2}(f(\gamma_{1}))-a_{2}(f(\gamma_{2})) 
\end{eqnarray*}
and denote the modulo ${\rm lk}(L_{f})$ reduction of $n(f,D)$ by $\bar{n}(f)$. It is clear that $\bar{n}(f)$ is also an ambient isotopy invariant of $f$. In particular, $\bar{n}(f)$ is a uniquely determined integer if ${\rm lk}(L_{f})=0$. In this case we denote $\bar{n}(f)$ by $n(f)$ simply. Then we have the following. 

\begin{Lemma}\label{nf_irr}
Let $f$ be a spatial handcuff graph. If $f$ is not irreducible, then for every choice of $D$, $n(f,D)=0$. In particular, if ${\rm lk}(L_{f})=0$, then $f$ is irreducible if $n(f)\neq 0$. 
\end{Lemma}

\begin{proof}
If $f$ is not irreducible, then $L_{f}$ is split and any $D$-sum of $L_{f}$ with respect to $f$ is the connected sum of $f(\gamma_{1})$ and $f(\gamma_{2})$. Recall that $a_{2}$ is additive under the connected sum of knots \cite{Kauffman87}. Thus we have that 
\begin{eqnarray*}
n(f,D)&=&a_{2}(K_{D}(f))-a_{2}(f(\gamma_{1}))-a_{2}(f(\gamma_{2}))\\
&=& a_{2}(f(\gamma_{1})\sharp f(\gamma_{2}))-a_{2}(f(\gamma_{1}))-a_{2}(f(\gamma_{2}))\\
&=& a_{2}(f(\gamma_{1}))+a_{2}(f(\gamma_{2}))-a_{2}(f(\gamma_{1}))-a_{2}(f(\gamma_{2}))\\
&=&0.
\end{eqnarray*}
Therefore we have the result. 
\end{proof}

For integers $r$ and $s$, let $f_{r,s}$ be the spatial handcuff graph as illustrated in Fig. \ref{handcuff_mn}, where the rectangles represented by $r$ and $s$ stand for $|r|$ full twists and $|s|$ full twists as illustrated in Fig. \ref{handcuff_mn2}, respectively. Note that the constituent $2$-component link $L_{f_{r,s}}$ is trivial. Then we have the following. 
\begin{figure}[htbp]
      \begin{center}
\scalebox{0.4}{\includegraphics*{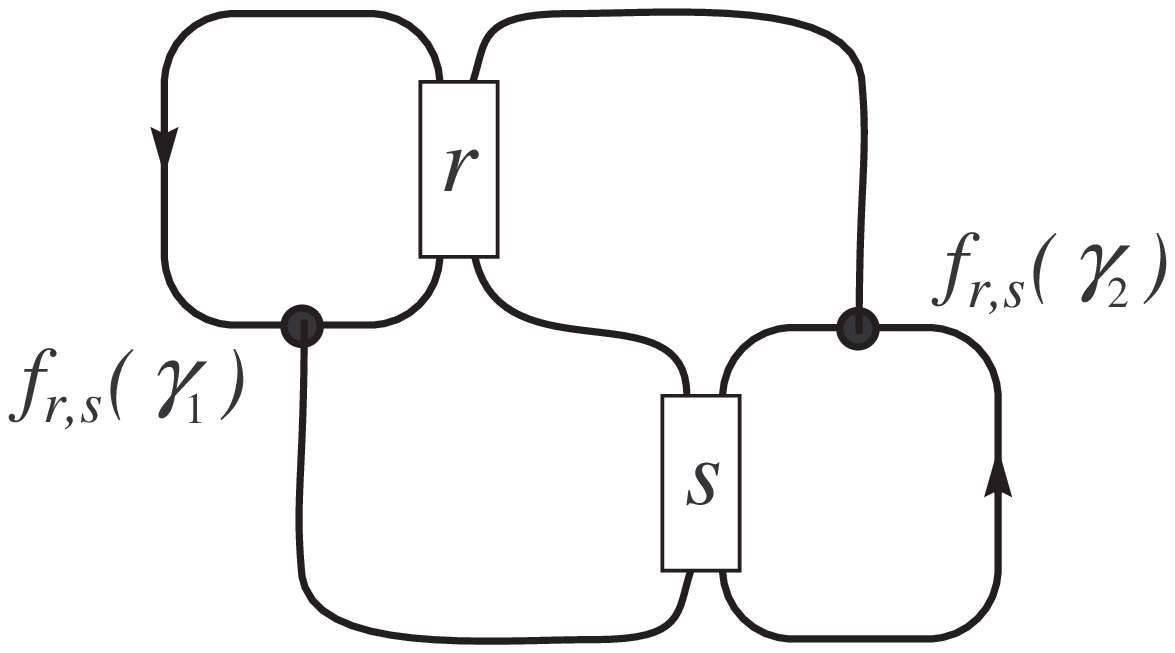}}
      \end{center}
   \caption{}
  \label{handcuff_mn}
\end{figure} 
\begin{figure}[htbp]
      \begin{center}
\scalebox{0.35}{\includegraphics*{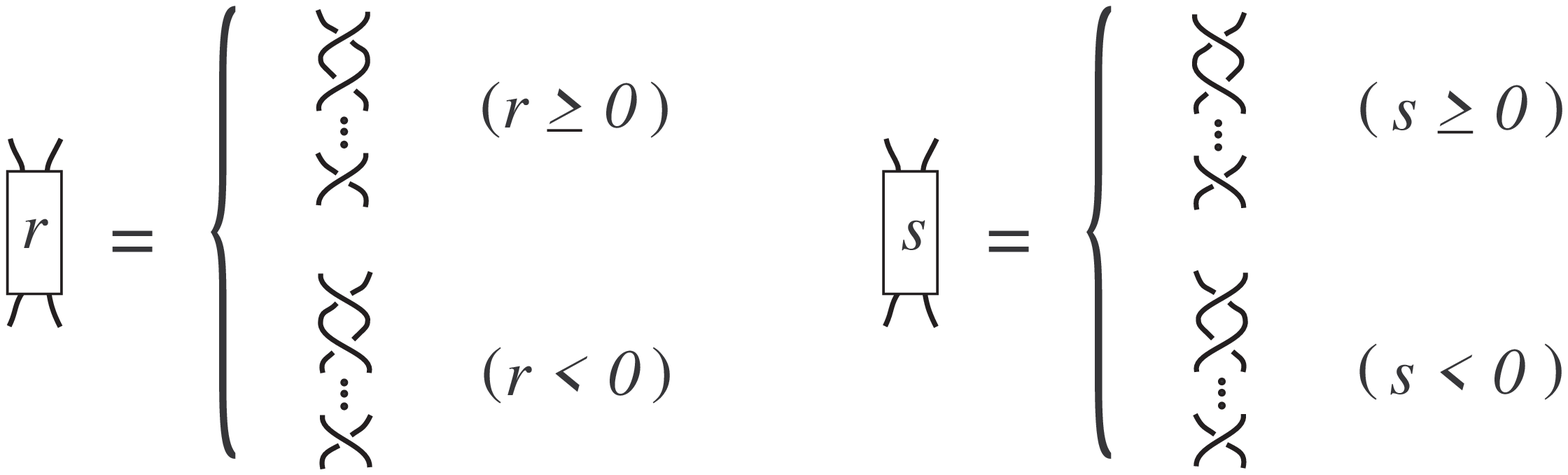}}
      \end{center}
   \caption{}
  \label{handcuff_mn2}
\end{figure} 

\begin{Lemma}\label{nf_irr_ex}
$n(f_{r,s})=2rs$. 
\end{Lemma}

\begin{proof}
Let $K_{+},\ K_{-}$ and $K_{0}$ be two oriented knots and an oriented $2$-component link which are identical except inside the depicted regions as illustrated in Fig. \ref{a2skein}. Then it is well known that $a_{2}(K_{+})-a_{2}(K_{-})={\rm lk}(K_{0})$ \cite{Kauffman83}. We consider the $D$-sum of $L_{f}$ with respect to $f_{r,s}$ and the skein tree started from $K_{D}(f_{r,s})$ as illustrated in Fig. \ref{handcuff_mn3}, where $\varepsilon=\pm 1$ is the usual sign of marked crossing. Note that $\varepsilon=1$ if $r>0$ and $-1$ if $r<0$. Then we have that 
\begin{eqnarray}
a_{2}(K_{D}(f_{r,s}))-a_{2}(J)&=&\varepsilon {\rm lk}(M_{1})=\varepsilon (r-\varepsilon +s),\label{mn1}\\
a_{2}(K_{D}(f_{r-\varepsilon,s}))-a_{2}(J)&=&\varepsilon {\rm lk}(M_{2})=\varepsilon (r-\varepsilon -s). \label{mn2}
\end{eqnarray}
Thus by (\ref{mn1}) and (\ref{mn2}) we have that 
\begin{eqnarray*}
a_{2}(K_{D}(f_{r,s}))-a_{2}(K_{D}(f_{r-\varepsilon,s}))=2\varepsilon s. 
\end{eqnarray*}
Hence we have that 
\begin{eqnarray}
a_{2}(K_{D}(f_{r,s}))&=&a_{2}(K_{D}(f_{r-\varepsilon,s}))+2\varepsilon s \nonumber\\
&=& a_{2}(K_{D}(f_{r-2\varepsilon,s}))+2\varepsilon s + 2\varepsilon s \nonumber\\
&\vdots&\nonumber\\
&=& a_{2}(K_{D}(f_{r-|r|\varepsilon,s}))+2\varepsilon |r| s\nonumber\\
&=& a_{2}(K_{D}(f_{0,s}))+2rs. \label{mn3}
\end{eqnarray}
It is easy to see that $f_{0,s}$ is trivial, namely $a_{2}(K_{D}(f_{0,s}))=0$. Therefore by (\ref{mn3}) we have that 
\begin{eqnarray*}
n(f_{r,s})=a_{2}(K_{D}(f_{r,s}))-a_{2}(f_{r,s}(\gamma_{1}))-a_{2}(f_{r,s}(\gamma_{2}))=2rs. 
\end{eqnarray*}
This completes the proof. 
\end{proof}
\begin{figure}[htbp]
      \begin{center}
\scalebox{0.5}{\includegraphics*{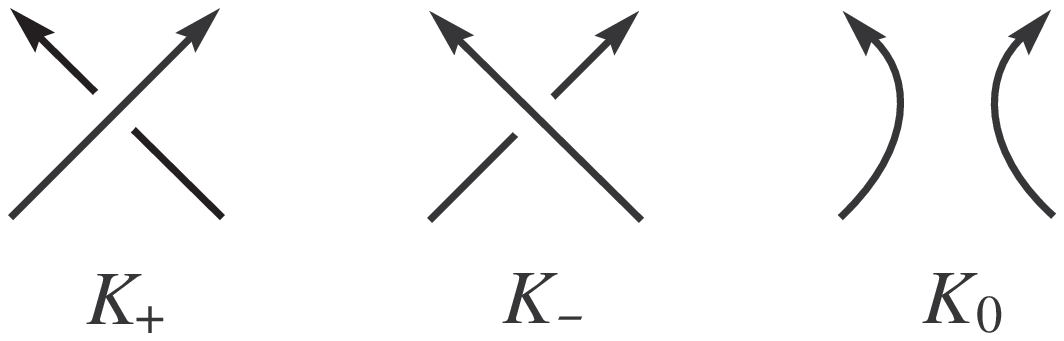}}
      \end{center}
   \caption{}
  \label{a2skein}
\end{figure} 
\begin{figure}[htbp]
      \begin{center}
\scalebox{0.375}{\includegraphics*{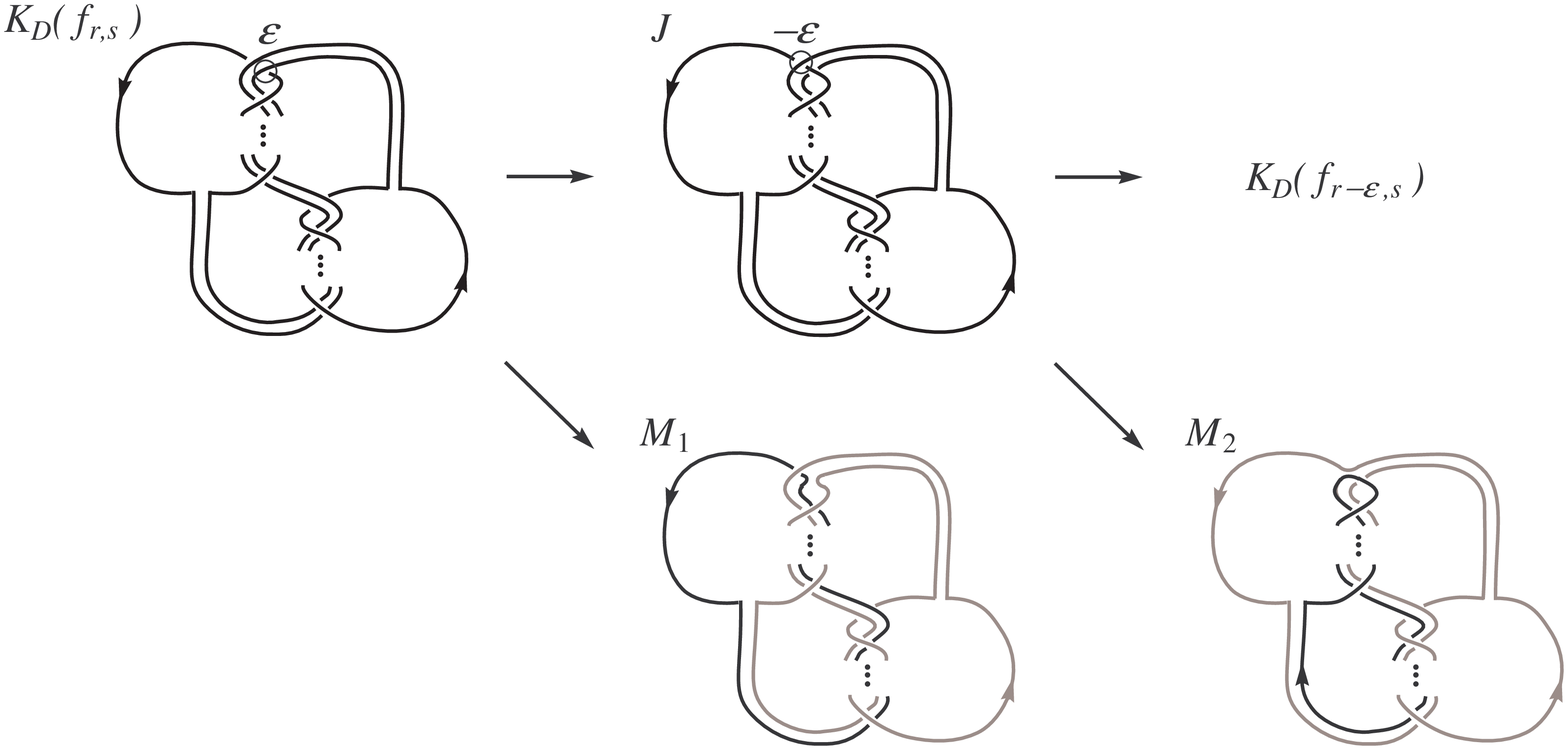}}
      \end{center}
   \caption{}
  \label{handcuff_mn3}
\end{figure} 

By Lemmas \ref{nf_irr} and \ref{nf_irr_ex}, we have that if $r,s\neq 0$ then $f_{r,s}$ is irreducible. Since $L_{f_{r,s}}$ is trivial, $\left\{f_{r,s}\right\}_{r,s\neq 0}$ is a family of {\it minimally knotted} spatial handcuff graphs, namely each $f_{r,s}$ is non-trivial and any of whose spatial proper subgraphs is trivial. 

\section{Proof of Theorem \ref{main}} 

Let $P_{4}$ be the oriented graph consists of four edges $e_{1},e_{2},e_{3},e_{4}$ and two loops $e_{5}, e_{6}$ as illustrated in Fig. \ref{DP4}. We denote the cycles $e_{5}$, $e_{1}\cup e_{2}$, $e_{3}\cup e_{4}$ and $e_{6}$ of $P_{4}$ by $c_{1},c_{2},c_{3}$ and $c_{4}$, respectively, and the subgraph $c_{1}\cup e_{i}\cup e_{j}\cup c_{4}\ (i=1,2,\ j=3,4)$ by $H_{ij}$. Note that $H_{ij}$ is homeomorphic to the graph $H$ which is illustrated in Fig. \ref{handcuff}. Let $f$ be a spatial embedding of $P_{4}$ with ${\rm lk}(f(c_{1}\cup c_{4}))=0$. Then we define $\xi(f)\in {\mathbb Z}$ by 
\begin{eqnarray*}
\xi(f)
&=&
\sum_{i,j}(-1)^{i+j}n(f|_{H_{ij}})\\
&=&n(f|_{H_{13}})-n(f|_{H_{14}})-n(f|_{H_{23}})+n(f|_{H_{24}}). 
\end{eqnarray*}
\begin{figure}[htbp]
      \begin{center}
\scalebox{0.4}{\includegraphics*{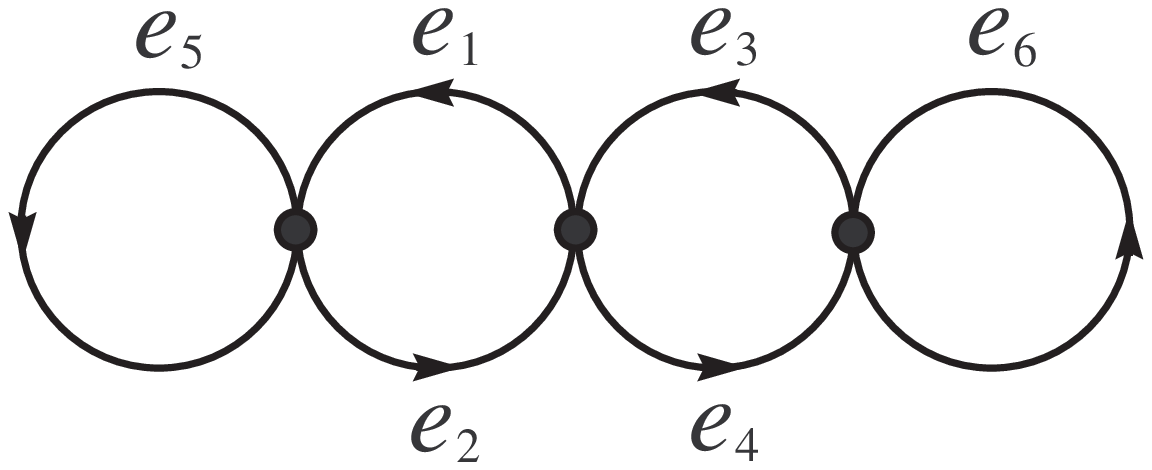}}
      \end{center}
   \caption{}
  \label{DP4}
\end{figure} 

\begin{Proposition}\label{homo_inv}
$\xi(f)$ is a Delta equivalence invariant of $f$. 
\end{Proposition}

Here, a {\it Delta equivalence} is an equivalence relation on spatial graphs which is generated by {\it Delta moves} and ambient isotopies, where a Delta move is a local move on a spatial graph as illustrated in Fig. \ref{Delta} \cite{Matveev87}, \cite{MN89}. It is shown in \cite{MT97} that a Delta equivalence coincides with a {\it (spatial graph-)homology}, which is an equivalence relation on spatial graphs introduced in \cite{Taniyama94}. Note that a Delta move preserves the linking number of each of the constituent $2$-component links. We show the following lemma needed to prove Proposition \ref{homo_inv}.
\begin{figure}[htbp]
      \begin{center}
\scalebox{0.45}{\includegraphics*{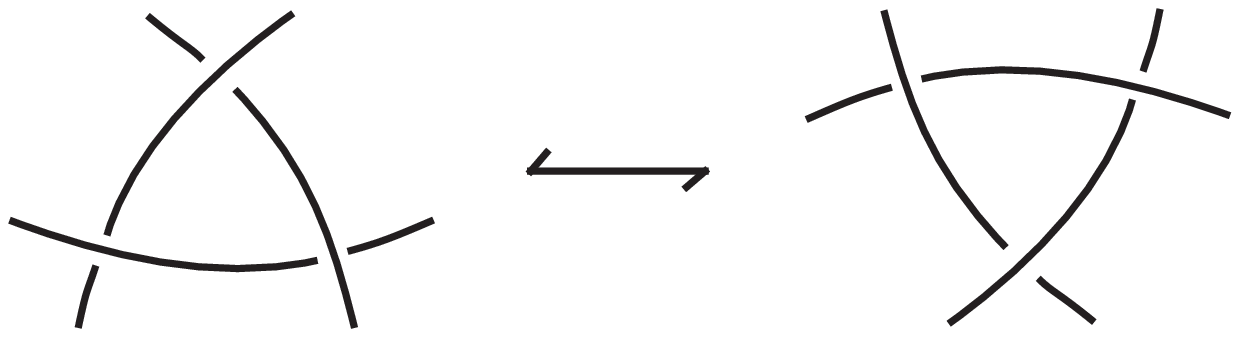}}
      \end{center}
   \caption{}
  \label{Delta}
\end{figure} 

\begin{Lemma}\label{handcuff_delta}
Let $g$ be a spatial handcuff graph with ${\rm lk}(L_{g})=0$ and $h$ a spatial handcuff graph obtained from $g$ by a single Delta move. Then we have the following. 
\begin{enumerate}
\item If either $g(\gamma_{1})$ or $g(\gamma_{2})$ does not appear in the Delta move as the strings, then $n(g)=n(h)$. 
\item If both $g(\gamma_{1})$ and $g(\gamma_{2})$ appear and $g(e)$ does not appear in the Delta move as the strings, then $n(g)-n(h)=\pm 1$. 
\item If all of $g(\gamma_{1})$, $g(\gamma_{2})$ and $g(e)$ appear in the Delta move as the strings, then $n(g)-n(h)=\pm 2$ or $0$. 
\end{enumerate}
\end{Lemma}

\begin{proof}
(1) If all of the three strings in such an intended Delta move belong to the same spatial edge (such a move is calld a {\it self Delta move}), we have the result because it is known that $n(g)$ is invariant under a self Delta move \cite[Theorem 2.1]{Nikkuni08}. Next we consider the case that at least one of the three strings in the Delta move belong to $g(e)$. If $g(\gamma_{j})$ does not appear in the Delta move as the strings ($j=1$ or $2$), then by applying the deformation on $g(H)$ as illustrated in Fig. \ref{Delta2} repeatedly, we can see that such a Delta move may be realized by self Delta moves on $g(\gamma_{i})$ ($i\neq j$). Therefore we have that $n(g)=n(h)$. 

\noindent
(2) If $h$ is obtained from $g$ by such an intended single Delta move, then we may consider $K_{D}(g)$ and $K_{D}(h)$ by the same $2$-disk $D$ so that $K_{D}(h)$ is obtained from $K_{D}(g)$ by a single Delta move. Then by the result of Okada \cite{Okada90} that if a knot $K_{1}$ is obtained from a knot $K_{2}$ by a single delta move then $a_{2}(K_{1})-a_{2}(K_{2})=\pm 1$, we have that $a_{2}(K_{D}(g))-a_{2}(K_{D}(h))=\pm 1$. On the other hand, since a Delta move is a {\it $3$-component Brunnian local move} \cite[\S 2, Examples (2)]{TY01}, we have that $g(\gamma_{i})$ and $h(\gamma_{i})$ are ambient isotopic ($i=1,2$). Hence we have that $n(g)-n(h)=a_{2}(K_{D}(g))-a_{2}(K_{D}(h))=\pm 1$. 

\noindent
(3) If $h$ is obtained from $g$ by such an intended single Delta move, then we may consider $K_{D}(g)$ and $K_{D'}(h)$ so that $K_{D'}(h)$ is obtained from $K_{D}(g)$ by twice Delta moves as illustrated in Fig. \ref{Delta3}. Thus by Okada's result as we said in the proof of (2), we have that $a_{2}(K_{D}(g))-a_{2}(K_{D'}(h))=\pm 2$ or $0$. Note that $g(\gamma_{i})$ and $h(\gamma_{i})$ are ambient isotopic ($i=1,2$) by Brunnian property of the Delta move. Hence we have that $n(g)-n(h)=a_{2}(K_{D}(g))-a_{2}(K_{D'}(h))=\pm 2$ or $0$. 
\end{proof}
\begin{figure}[htbp]
      \begin{center}
\scalebox{0.45}{\includegraphics*{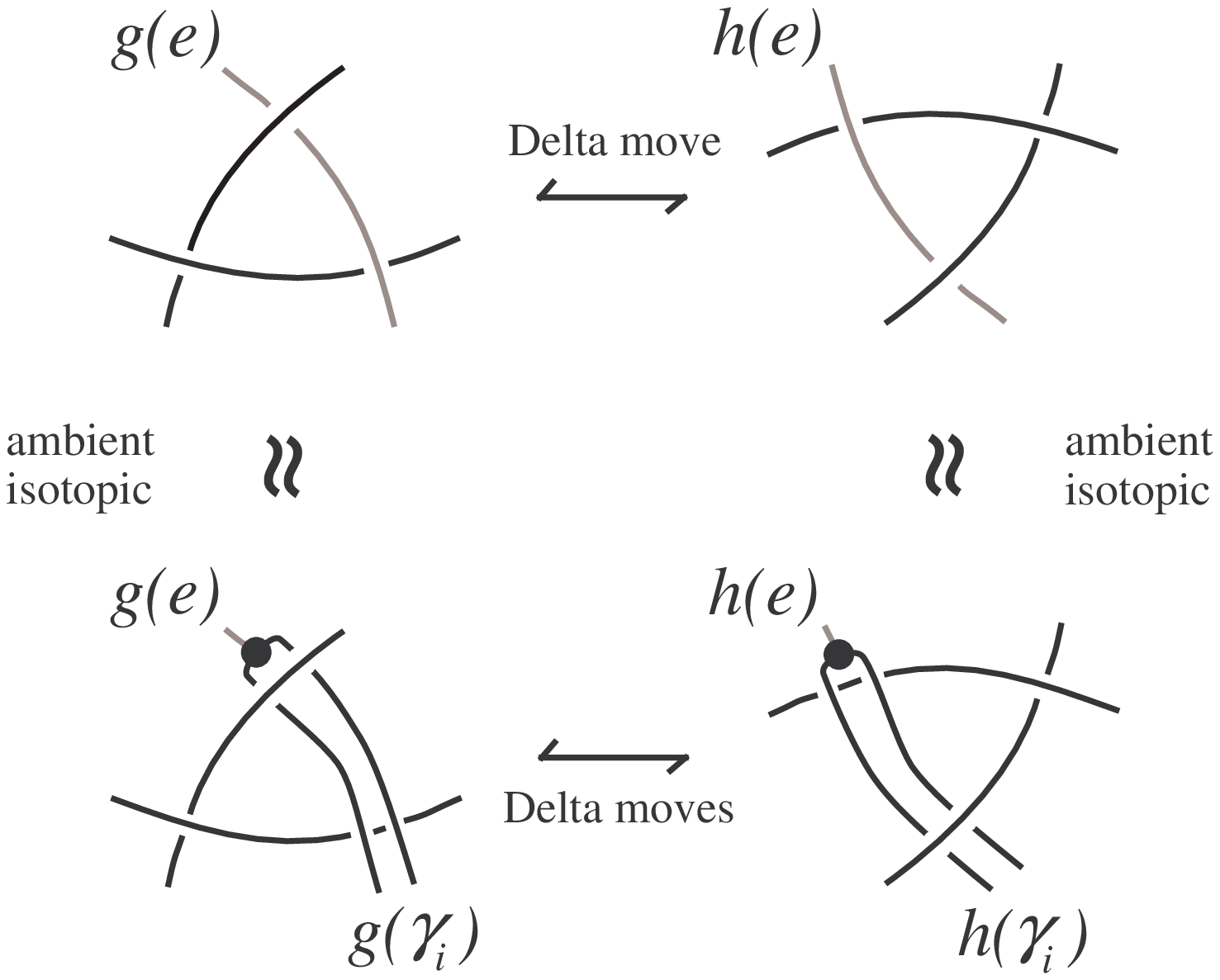}}
      \end{center}
   \caption{}
  \label{Delta2}
\end{figure} 
\begin{figure}[htbp]
      \begin{center}
\scalebox{0.45}{\includegraphics*{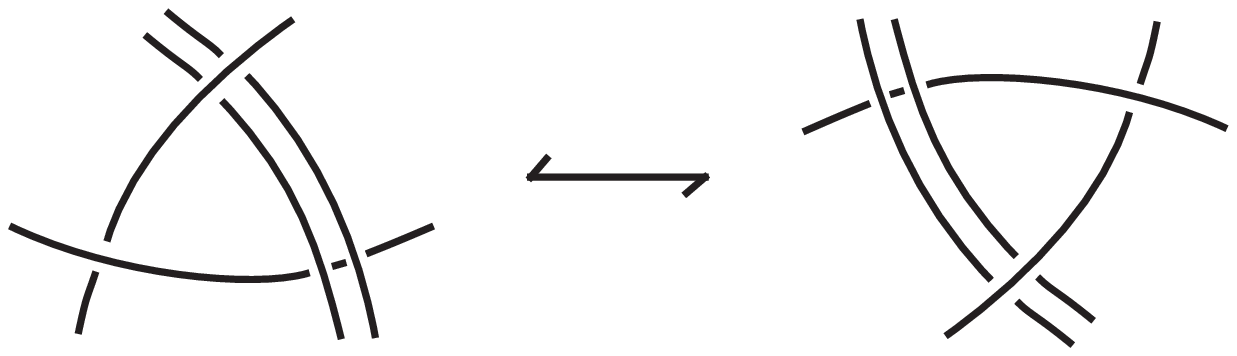}}
      \end{center}
   \caption{}
  \label{Delta3}
\end{figure} 

\begin{proof}[Proof of Proposition \ref{homo_inv}.] 
Let $f'$ be a spatial handcuff graph which is obtained from $f$ by a single Delta move. It is sufficient to show that $\xi(f)=\xi(f')$. If either $f(e_{5})$ or $f(e_{6})$ does not appear in the Delta move as the strings, then by Lemma \ref{handcuff_delta} (1) and Brunnian property of the Delta move, we have that $n(f|_{H_{ij}})=n(f'|_{H_{ij}})$ for any $i=1,2$ and $j=3,4$. Hence $\xi(f)=\xi(f')$. If none of $f(e_{1})$, $f(e_{2})$, $f(e_{3})$ or $f(e_{4})$ appears in the Delta move as the strings, then by Lemma \ref{handcuff_delta} (2) we have that $n(f|_{H_{ij}})-n(f'|_{H_{ij}})=\pm 1\ (i=1,2,\ j=3,4)$. Note that if two oriented knots differ by a single Delta move then the variation of $a_{2}$ of them is determined by only the order of strings in the Delta move and their orientations arised from following the knot along the orientation (cf. \cite{OT99}, \cite[Theorem 6]{Yamada00}). Hence we have that $n(f|_{H_{ij}})-n(f'|_{H_{ij}})=1\ (i=1,2,\ j=3,4)$ or $n(f|_{H_{ij}})-n(f'|_{H_{ij}})=-1\ (i=1,2,\ j=3,4)$. Then we have that
\begin{eqnarray*}
\xi(f)-\xi(f')
= \sum_{i,j}(-1)^{i+j}\left\{n(f|_{H_{ij}})-n(f'|_{H_{ij}})\right\}
= (\pm 1)\sum_{i,j}(-1)^{i+j}
= 0. 
\end{eqnarray*}
Finally we consider the case that all of $f(e_{5})$, $f(e_{6})$ and $f(e_{k})$ appear in the Delta move as the strings ($k=1,2,3,4$). It is sufficient to check the case of $k=1$. Then by Lemma \ref{handcuff_delta} (3) and the note of the variation of $a_{2}$ of two oriented knots differ by a single Delta move as we said before, it holds that any one of $n(f|_{H_{1j}})-n(f'|_{H_{1j}})=2\ (j=3,4)$, $n(f|_{H_{1j}})-n(f'|_{H_{1j}})=-2\ (j=3,4)$ or $n(f|_{H_{1j}})-n(f'|_{H_{1j}})=0\ (j=3,4)$. Note that $f|_{H_{2j}}$ and $f'|_{H_{2j}}$ are ambient isotopic for any $j=3,4$ by Brunnian property of the Delta move. Then we have that
\begin{eqnarray*}
\xi(f)-\xi(f')
= n(f|_{H_{13}})-n(f'|_{H_{13}})
-\left\{n(f|_{H_{14}})-n(f'|_{H_{14}})\right\}
= 0.
\end{eqnarray*}
Therefore we have the desired conclusion. 
\end{proof}

\begin{Lemma}\label{keylemma}
Let $f$ be a spatial embedding of $P_{4}$ with ${\rm lk}(f(c_{1} \cup c_{4}))=0$. Then we have that ${\rm lk}(f(c_{1} \cup c_{3})){\rm lk}(f(c_{2} \cup c_{4}))\neq 0$ if and only if $\xi(f)\neq 0$.
\end{Lemma}

\begin{proof}
It is known that two spatial embeddings of a planar graph are Delta equivalent if and only if their corresponding constituent $2$-component links have the same linking number \cite{SSY97}. Therefore we have that if ${\rm lk}(f(c_{1} \cup c_{3}))=s$ and ${\rm lk}(f(c_{2} \cup c_{4}))=r$, then $f$ is Delta equivalent to the spatial embedding $h_{r,s}$ of $P_{4}$ as illustrated in Fig. \ref{handcuff_mn4}, where the rectangles represented by $r$ and $s$ stand for $|r|$ full twists and $|s|$ full twists as illustrated in Fig. \ref{handcuff_mn2}, respectively. Note that $h_{r,s}|_{H_{13}}$, $h_{r,s}|_{H_{14}}$ and $h_{r,s}|_{H_{23}}$ are trivial spatial handcuff graphs. Then by combining Proposition \ref{homo_inv} and Lemma \ref{nf_irr_ex}, we have that 
\begin{eqnarray*}
\xi(f)=\xi(h_{r,s})= n(h_{r,s}|_{H_{24}})=2rs. 
\end{eqnarray*}
This implies the result. 
\end{proof}
\begin{figure}[htbp]
      \begin{center}
\scalebox{0.4}{\includegraphics*{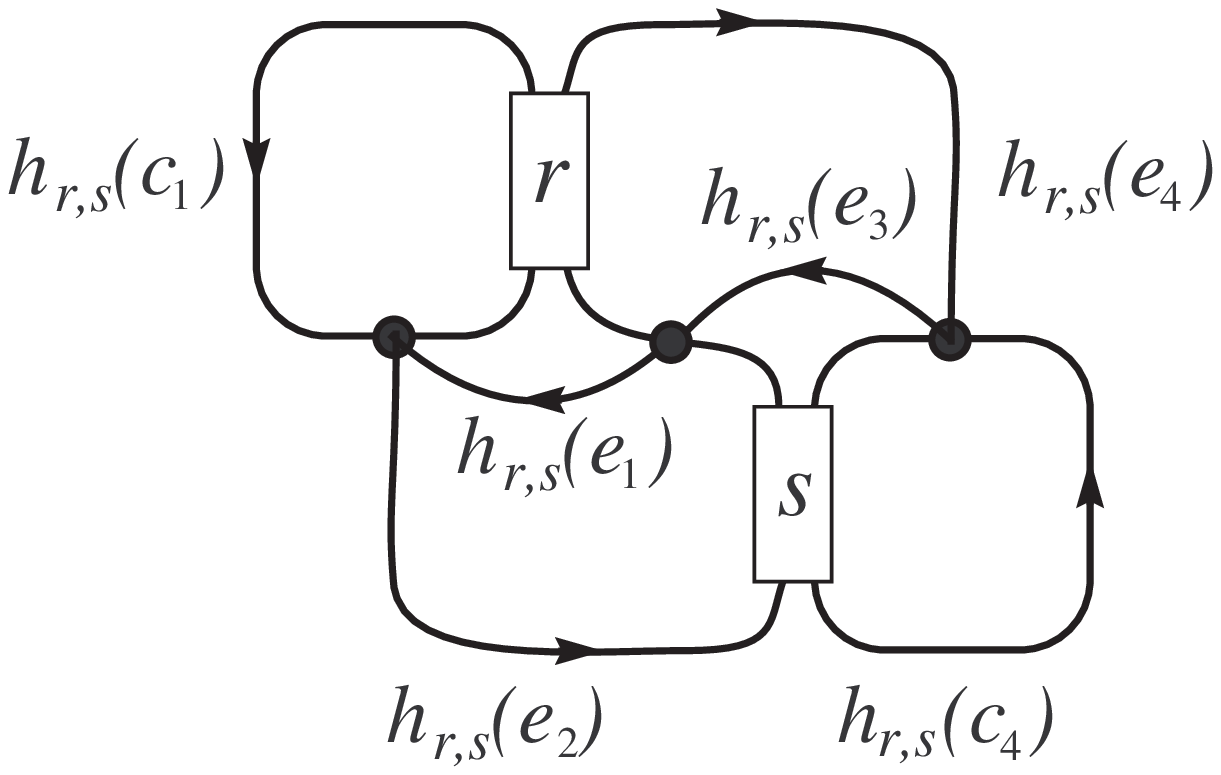}}
      \end{center}
   \caption{}
  \label{handcuff_mn4}
\end{figure} 

Note that if $\xi(f)\neq 0$ then there must exist a subgraph $H_{ij}$ such that $n(f|_{H_{ij}})\neq 0$. Then by Lemma \ref{nf_irr} we have that $f|_{H_{ij}}$ is irreducible. 

\begin{proof}[Proof of Theorem \ref{main}.] 
Let $f$ be a spatial embedding of $P {*}_{3} P'$. Note that every spatial embedding of a graph in the Petersen family contains a $2$-component link with odd linking number \cite{CG83}, \cite{TY01}. Hence there exists a pair of two disjoint cycles $\gamma_{1},\gamma_{2}$ of $P$ such that ${\rm lk}(f(\gamma_{1}\cup \gamma_{2}))\neq 0$, and there exists a pair of two disjoint cycles $\gamma'_{1},\gamma'_{2}$ of $P'$ such that ${\rm lk}(f(\gamma'_{1}\cup \gamma'_{2}))\neq 0$. If there exist two edges $e_{i}$ and $e_{j}$ ($1\le i<j \le 3$) each of which connects $\gamma_{l}$ with $\gamma'_{k}$ for some $l$ and $k$, by Bowlin-Foisy's argument \cite[Lemma 3]{BF04}, we have that $f(\gamma_{1}\cup \gamma_{2}\cup \gamma'_{1}\cup \gamma'_{2}\cup e_{i}\cup e_{j})$ contains a non-splittable $3$-component link. If for any $l$ and $k$ there do not exist two edges $e_{i}$ and $e_{j}$ ($1\le i<j \le 3$) each of which connects $\gamma_{l}$ with $\gamma'_{k}$, then we may assume that $e_{1}$ connects $\gamma_{1}$ with $\gamma'_{1}$, $e_{2}$ connects $\gamma_{1}$ with $\gamma'_{2}$ and $e_{3}$ connects $\gamma_{2}$ with $\gamma'_{2}$ without loss of generality. If $f(\gamma_{2}\cup \gamma'_{1})$ is non-splittable, then $f(\gamma_{1}\cup \gamma_{2}\cup \gamma'_{1})$ is a non-splittable $3$-component link. If $f(\gamma_{2}\cup \gamma'_{1})$ is split, then let us consider the subgraph $F=\gamma_{1}\cup \gamma_{2}\cup \gamma'_{1}\cup \gamma'_{2}\cup e_{1}\cup e_{2}\cup e_{3}$ of $P {*}_{3} P'$. We denote the graph obtained from $F$ by contracting $e_{1},\ e_{2}$ and $e_{3}$ by $F/e_{1}/e_{2}/e_{3}$. Note that $F/e_{1}/e_{2}/e_{3}$ is homeomorphic to $P_{4}$ which is illustrated in Fig. \ref{DP4}. Let $\bar{f}$ is a spatial embedding of $F/e_{1}/e_{2}/e_{3}$ naturally induced from $f|_{F}$. Then by applying Lemmas \ref{keylemma} and \ref{nf_irr} to $\bar{f}$, we have that $\bar{f}(F/e_{1}/e_{2}/e_{3})$ contains an irreducible spatial handcuff graph whose constituent $2$-component link is split. Then it is clear that $f(F)$ also contains a spatial handcuff graph with the same spatial graph type as above. This implies that $f$ contains an irreducible spatial handcuff graph whose constituent $2$-component link is split. This completes the proof. 
\end{proof}

\begin{Example}\label{ex}
{\rm 
Let $f$ be the spatial embedding of $K_{6} {*}_{3} K_{6}$ as illustrated in Fig. \ref{example}. Then we can see that each of the spatial handcuff graphs contained in $f$ is trivial or ambient isotopic to the spatial handcuff graph as illustrated in Fig. \ref{Hopf_handcuff}. Namely $f$ does not contain an irreducible spatial handcuff graph whose constituent $2$-component link is split. But $f$ contains exactly one non-splittable $3$-component link. Next, let $g$ be the spatial embedding of $K_{6} {*}_{3} K_{6}$ as illustrated in Fig. \ref{example} which is obtained from $f$ by a single crossing change. Then we can see that each of the constituent knots of $g$ is trivial and each of the constituent $n$-component links of $g$ is split for $n\ge 3$. But $g$ contains exactly one irreducible spatial handcuff graph whose constituent $2$-component link is split. 
}
\end{Example}
\begin{figure}[htbp]
      \begin{center}
\scalebox{0.375}{\includegraphics*{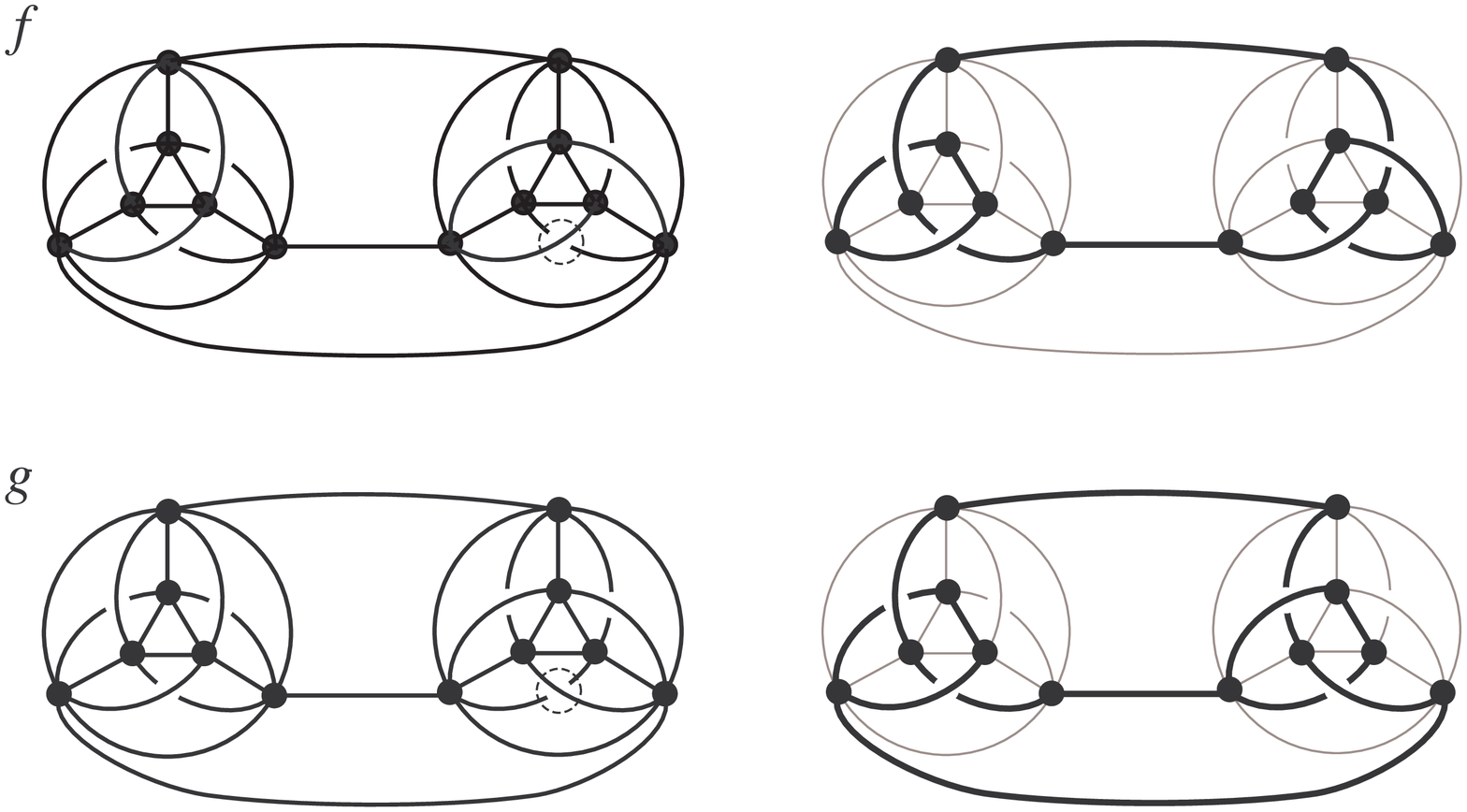}}
      \end{center}
   \caption{}
  \label{example}
\end{figure}
\begin{figure}[htbp]
      \begin{center}
\scalebox{0.4}{\includegraphics*{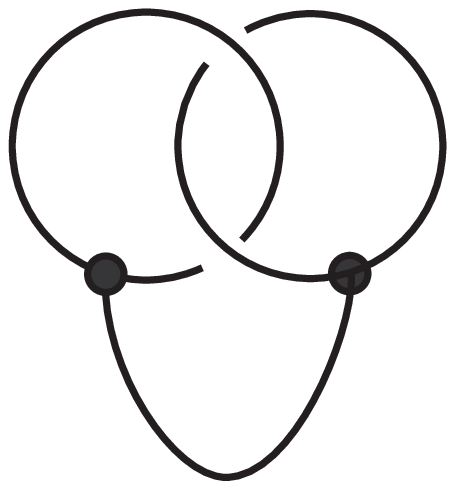}}
      \end{center}
   \caption{}
  \label{Hopf_handcuff}
\end{figure}

\section*{Acknowledgment}

The author is grateful to the referee for his or her comments. 

{\normalsize
}

\end{document}